%% file: paper4.tex
\documentclass[reqno]{amsart}
\usepackage{amsaddr}
\usepackage[utf8]{inputenc}
\usepackage[T1]{fontenc}
\usepackage{amsmath, amsthm, amssymb, xcolor, graphicx, lineno}
\usepackage[colorlinks=true, allcolors=blue]{hyperref}

\newtheorem{theorem}{Theorem}
\newtheorem{lemma}{Lemma}
\newtheorem{corollary}{Corollary}
\newtheorem{remark}{Remark}
\newtheorem{definition}{Definition}
\newcommand{\norm}[1]{\left\lVert#1\right\rVert}

\title[DNWR for heterogeneous heat]{Dirichlet--Neumann waveform relaxation for heterogeneous heat equations: fully discrete $l^2$ analysis}
\author[P. Birken, N. Kotarsky]{Philipp Birken$^{{\lowercase\mathrm{a}}}$, Niklas Kotarsky$^{{\lowercase\mathrm{a}}}$}

\address[Lund]{
	$\,^{\lowercase\mathrm{a}}$ Centre for Mathematical Sciences, 
	Lund University\\ 
	Box 188, 22100 Lund,
	Sweden}



\begin{document}
	
	\begin{abstract}
		We consider two coupled linear heat equations on different spatial domains that interact through a lower dimensional interface. This models conjugate heat transfer. The problem is solved using Dirichlet--Neumann waveform relaxation. This allows the subproblems to be solved using separate codes, a so called partitioned approach. Our overall goal is to develop more efficient partitioned methods, and to this end, we want reliable error estimates. 
		
		Here, we use an exponentially weighted Fourier technique to derive new error estimates in $l^2$ for finite time $T$ in the fully discrete setting. These describe both linear and superlinear behavior. We show that the fully discrete estimate is close to a previously obtained time discrete estimate and independent of $\Delta x_1$ and $\Delta x_2$ when the CFL number in the subsolvers is large. We also show that the convergence behaviour depends on the ratio $\Delta x_1/\Delta x_2$ when the CFL number is small.
		
		Our numerical experiments show that the fully discrete estimate is accurate across a wide range of grid sizes $\Delta x_1, \Delta x_2$, time step sizes $\Delta t$ and $T$.  
	\end{abstract}
	
	\maketitle


\input{introduction.tex}
\input{modelProblem.tex}
\input{fullyDiscreteEstimates.tex}	
\input{analysisOfFullyDiscEst.tex}

\input{experiments.tex}

\section{Conclusion}

We presented $l^2$ error estimates for the Dirichlet-Neumann waveform relaxation method applied to two coupled heat equations discretized with implicit Euler in time and finite differences in space. Our analysis is based on the exponentially weighted Fourier technique, introducing an auxiliary weight parameter $r$ to derive error estimates that depend on the time length $T$, the time step size $\Delta t$, and the grid sizes $\Delta x_1$ and $\Delta x_2$. 


 
From the error estimate, one can distinguish different behavior of the iteration if the CFL numbers in the subsolvers are small or large. For large CFL numbers, the behavior is the same as that of the time discrete counterpart and does not depend on $\Delta x_1$ and $\Delta x_2$. This situation has been analyzed in \cite{timeDiscArticle}. In contrast, for small CFL numbers, we expect linear convergence that depends on $\Delta x_1/\Delta x_2$, meaning that there is a dependence on the spatial mesh.
 
 
 
These theoretical findings are validated by numerical experiments, which demonstrate that the fully discrete estimates accurately capture the observed convergence behavior of the DNWR iteration, across a wide range of $T$, $\Delta t$, $\Delta x_1$ and $\Delta x_2$. 

\bibliographystyle{siam}
\bibliography{ourbibliography}
	
\end{document}

%% file: introduction.tex
\section{Introduction}

We consider two coupled linear heat equations on different spatial domains that interact through a lower dimensional interface. This setup models conjugate heat transfer, with applications ranging from the cooling of rocket motors \cite{kowollik2013} and turbine blades to gas quenching, an industrial heat-treatment process for metals \cite{vynnycky1998, monbir:18}. In addition, conjugate heat transfer also plays an important part in climate models, being, for example, the dominant mechanism in the interaction between ocean and atmosphere \cite{clement2022, DD29}. 

These types of problems are often solved using a partitioned approach, in which separate sub-solvers operate on each subdomain and interact only through boundary conditions at the interface. Depending on the choice of interface conditions, this leads to methods such as Dirichlet-Neumann or Robin-Robin coupling. A major advantage of the partitioned approach is that the discretization in each sub-solver can be chosen independently.

Independent time discretization of the sub-solvers can be achieved by so called Waveform Relaxation. This was originally introduced in \cite{lelarasmee1982} for the parallel simulation of large electronic circuits, modeled as large systems of ODEs. The basic idea is to solve coupled ODEs iteratively by alternately solving each subsystem while using data from the other subsystems. Waveform relaxation has since been combined with domain decomposition techniques to enable parallel solution of large PDE systems, see \cite[Chapter 3]{waveformBook} for a general introduction. Moreover, waveform relaxation methods are also used in multiphysics simulations. In the context of thermal FSI, they have been combined with time-adaptive sub-solvers in \cite{meimonbir:23}. Additionally, the coupling strategy employed in many earth system models in climate science can be interpreted as a Schwarz waveform relaxation that is terminated after the first iteration \cite{schuller2025}.

Achieving fast convergence in waveform relaxation methods often requires an appropriate choice of relaxation parameters. Suitable choices may lead to {\sl superlinear convergence} on finite time intervals. This property is particularly advantageous in long time simulations. In such simulations, the global time horizon is typically decomposed into smaller time intervals, and waveform relaxation iterations are performed within each time interval until a stopping criterion is satisfied before marching to the next time interval. Therefore, understanding how convergence depends on relaxation parameters, physical parameters, and the length of the time interval, is a central topic in the analysis of waveform relaxation methods.

In addition, the convergence behaviour is influenced by the choice of time and space discretizations. In both FSI and climate modeling, coarse discretizations that do not fully resolve the underlying physics are often employed. In FSI, this commonly occurs when modeling turbulence \cite{moin1998}, while for climate science it arises, for example, in Earth system models, where spatial meshes with horizontal resolutions of $50-100 km$ are in use. This raises the question of how coarse temporal and spatial resolution affect convergence behaviour. In addition, in both climate science and FSI, it is common to use different grid sizes in the different domains, raising the question how mismatches in spatial resolution affect the convergence behaviour. 


Several convergence analysis techniques for waveform relaxation methods have been developed~\cite{arnoult2023, cllebl:22, engstrom2024, gander2007, haynes2020, ledebl:13, monbir:18}. We list here some references that contain error estimates in the continuous, time discrete, space discrete and fully discrete setting. In the continuous setting, the authors in~\cite{ledebl:13} derive linear $L^2$ error estimates for WR applied to heterogeneous advection--diffusion equations on unbounded space-time domains, using Fourier analysis and the Plancherel theorem. In~\cite{gakwma:16}, the authors investigate DNWR and NNWR methods applied to homogeneous heat equations on finite time intervals. Based on the Laplace transform and kernel estimates, they derive superlinear $L^{\infty}$ error estimates using special relaxation parameters, $\theta=1/2$ for DNWR and $\theta=1/4$ for NNWR. In~\cite{DD28}, we extend this analysis of DNWR to heterogeneous heat equations, and also derive superlinear $L^{\infty}$ error estimates for a material-dependent relaxation parameter $\theta$. In~\cite{timeDiscArticle}, we use an exponentially weighted Fourier technique to derive error estimates in $L^2$ for finite $T$ and arbitrary relaxation parameters. We also show that the estimates yield superlinear convergence for a material-dependent relaxation parameter.

A time discrete analysis for heterogeneous advection-diffusion equations is presented in~\cite{clement2022}, where Clement et al. apply both implicit Euler and Runge-Kutta schemes, and derive linear $l^2$ error estimates using the $z$-transform on unbounded space-time domains. In~\cite{timeDiscArticle} we extend our continuous analysis for DNWR applied to heterogeneous heat equations to the time-discrete setting for implicit Euler. We use the exponentially weighted discrete time Fourier transform to obtain $l^2$ error estimates for finite $T$, allowing us to identify the time step sizes $\Delta t$ that recover the superlinear convergence behaviour of the continuous estimate.

The space discrete case has been analyzed for a reaction diffusion equation on unbounded spatial and temporal domains using finite differences in \cite{wu2014} and using both finite difference and finite volume discretizations in \cite{clement2022}. In both works, a Fourier transform in time is employed and the resulting difference equations are solved analytically, yielding linear error estimates. 

Lastly, the fully discrete case is analyzed in \cite{clement2022} for implicit Euler as well as Runge-Kutta methods in time, and finite difference or finite volume discretizations in space. In this analysis, Clement et. al. employ the $z$-transform in time and solve the resulting difference equations analytically, yielding linear error estimates. They also show that the discretization of the interface condition affects the convergence rate. However, they do not discuss how different spatial grid sizes in the domains affect convergence. 

The impact of different grid sizes is examined in \cite{monbir:18} for the special case in which both subsolvers do one implicit Euler time step. There, Monge and Birken analyze two coupled heteregeneous heat equations in one dimension, discretized using either a finite element or a finite volume method in space, yielding a convergence factor. They show that the CFL numbers of the two subsolvers, as well as the ratio of the grid sizes $\Delta x_1$ and $\Delta x_2$, have a significant impact on the convergence factor. In particular, they find that, if the two subsolver's CFL number is small, which is the case when the time step is small, then the convergence factor is proportional to the ratio $\Delta x_1 / \Delta x_2$. This result is counterintuitive, as it implies that refining the discretization does not yield a convergence rate that converges to that of the continuous iteration, which does not depend on the discretization. It is also not clear how the convergence factor for one time step from \cite{monbir:18} generalizes to multiple time steps.

Despite this body of work, to the best of our knowledge, there are currently only linear fully discrete error estimates. Furthermore, the effect of the spatial meshes is not fully understood. This is of particular interest in both climate science and thermal FSI, where large and different grid sizes are often employed in different domains. Finally, our understanding of the relation of the continuous, semidiscrete and fully discrete iterations is lacking. 

The present paper aims to answer these questions for two coupled linear heat equations with different material parameters discretized in time with implicit Euler and by finite differences in space. To analyze the convergence of DNWR applied to such problems on finite time intervals, we use the exponentially weighted Fourier transform proposed in~\cite{Kwok2025} and used in \cite{timeDiscArticle}. Using this framework, we derive $l^2$ estimates for finite time lengths $T$ in the fully discrete case. 

We analyze the convergence behaviour and show that the spatial discretization mainly affects large Fourier frequencies. In addition, we show that the ratio $\Delta x_1 / \Delta x_2$ is important when the CFL numbers of the subsolvers are small. We also show that the fully discrete estimate is close to the time discrete one, and thus independent of $\Delta x$, when the CFL numbers of the subsolvers are large. 

The remainder of the paper is organized as follows. In Section~\ref{sec:fullydisc}, we present a convergence analysis of DNWR applied to two coupled heat equations in the fully discrete setting. Section~\ref{sec:3} compares the fully discrete estimate with the continuous and time discrete counterparts from~\cite{timeDiscArticle}. Finally, in Section~\ref{sec:numExp}, we validate the theoretical results numerical experiments using realistic material parameters for air and steel, as well as for atmosphere and ocean. 

%% file: modelProblem.tex
\section{The Model problem}\label{sec:fullydisc}

We consider two coupled linear heat equations in one dimension as our model problem. Let $Q_1 = (-a,0)$ and $Q_2 = (0,b)$ with $a>0,b>0$ denote the two spatial domains, let $T>0$ be the time length and $u_j$ denote the temperature in each domain. The coupled problem is then given by \begin{equation}\label{eq:CoupledSystem}
	\begin{aligned}
		\alpha_1 \partial_t u_1 -\lambda_1 \partial_{xx} u_1 &= f_1& &\text{in } Q_1,&
		\alpha_2 \partial_t u_2 -\lambda_2 \partial_{xx} u_2 &= f_2&  &\text{in } Q_2,\\
		u_1(-a, t) &= g_l(t)& &\text{in } (0,T],&	
		u_1(b, t) &= g_r(t)&  &\text{in } (0,T],\\
		u_1(x, 0) &= u_{0, 1}(x)& &\text{in } (-a,0),& 
		u_2(x, 0) &= u_{0, 2}(x)& &\text{in } (0,b),
	\end{aligned}
\end{equation}
where $\alpha_i$ is the respective heat capacity, $\lambda_i$ the heat conductivity, $f_i$ the external force, $u_{0,i}$ the initial condition and $ g_i$ the initial condition. To close the system we also use the continuity of temperature and heat-flux on the interface as coupling conditions, yielding
\begin{equation*}
	u_1(0, t) = u_2(0, t) \ \text{ in } (0,T], \quad 
	-\lambda_1 \partial_x u_1(0, t) = -\lambda_2 \partial_x u_2(0, t) \ \text{ in } (0,T].
\end{equation*}
To solve the problem, we make use of DNWR. For a given initial guess $h^0(t)$ and the iteration index $k = 1, 2, \ldots$, the DNWR iteration is given by
\begin{equation}\label{eq:DNWR}
	\begin{aligned}
		\alpha_1 \partial_t u_1^{k+1} - \lambda_1 \partial_{xx} u_1^{k+1} &= f_1 \ \text{ in } Q_1,&
		\alpha_2 \partial_t u_2^{k+1} - \lambda_2 \partial_{xx} u_2^{k+1} &= f_2 \ \text{ in } Q_2, \\
		u_1^{k+1}(-a, t) &= g_l(t), \ &
		-\lambda_1 \partial_x u_1^{k+1}(0, t) &= -\lambda_2 \partial_x u_2^{k+1}(0, t), \\
		u_1^{k+1}(0, t) &= h^{k}(t), &
		u_1^k(b, t) &= g_r(t), \\
		u_1^{k+1}(x, 0) &= u_{0, 1}(x), &
		u_2^{k+1}(x, 0) &= u_{0, 2}(x),
	\end{aligned}
	\tag{DNWR}
\end{equation}
which is combined with the relaxation step: 
\begin{equation*}
	h^{k+1}(t) = \theta u_2^{k+1}(0,t) + (1-\theta) h^k(t),
\end{equation*}
where $\theta \in [0,1]$ denotes the relaxation parameter. 

Error estimates in $L^2$ are derived in \cite{timeDiscArticle} using an exponentially weighted Fourier transform.
Denoting the Fourier frequency by $\omega$ and the weight parameter by $r$, the continuous error estimate in
\cite{timeDiscArticle} is given by \begin{align} \label{eq:contEst}
	\|h^k\|_{L^2(0,T)} \leq \min_{r \geq 0} e^{rT} \max_{\omega \in \mathbb{R}} \left|\hat{\rho}(r + i \omega)\right|^{k} \norm{h^0}_{L^2(0,T)},
\end{align}  where $\hat{\rho}$ denotes the Fourier transformed convergence factor
\[\hat{\rho}(s) = 1-\theta -\theta \sqrt{\frac{\alpha_1\lambda_1}{\alpha_2\lambda_2}}\coth\left(\sqrt{\frac{\alpha_1 s}{\lambda_1}}a\right)\tanh\left(\sqrt{\frac{\alpha_2 s}{\lambda_2}}b\right), \quad s\in \mathbb{C}.\] 
In \cite{timeDiscArticle}, it is also shown that the continuous estimate \eqref{eq:contEst} converges superlinearly for finite $T$ and $\theta = \sqrt{\alpha_1\lambda_1}/(\sqrt{\alpha_1\lambda_1} + \sqrt{\alpha_2\lambda_2}) =: \theta^*$.

To obtain the {\it time discrete iteration}, we discretize the two coupled heat equations using $N$ implicit Euler steps of size $\Delta t$. Denote by $u^n$ the approximation of the solution $u$ at time $t_n$ and $\boldsymbol{u} = (u^1, \ldots, u^N)^\top$. For the initial guess ${\bf h}^0$, the time discrete iteration reads: For $n=0, \ldots, N-1$, let  \begin{equation}\label{eq:DNWRIE}
	\begin{aligned}
		\frac{\alpha_1}{\Delta t} (u_1^{n+1, k+1} - u_1^{n, k+1}) - \lambda_1 \partial_{xx} u_1^{n+1, k+1} &= f_1^{n+1} \ \text{ in } \ (-a,0) , \\
		u_1^{n+1, k+1}(-a) = g_l^{n+1}, \quad 
		u_1^{n+1, k+1}(0) &= h^{n+1, k},\\
		u_1^{0, k+1}(x) &= u_{0, 1}(x) \\
		\frac{\alpha_2}{\Delta t} (u_2^{n+1, k+1}-u_2^{n, k+1})-\lambda_2 \partial_{xx} u_2^{n+1, k+1} &= f_2^{n+1} \ \text{ in } \ (0,b),\\
		u_2^{n+1,k+1}(b) = g_r^{n+1}, \quad
		-\lambda_1 \partial_x u_1^{n+1, k+1}(0) &= -\lambda_2 \partial_x u_2^{n+1, k+1}(0), \\
		u_2^{0, k+1}(x) &= u_{0, 2}(x),\\
		h^{n+1, k+1} := (1-\theta) h^{n, k} + \theta &u^{n+1, k+1}_2(0),  \quad \theta\in(0, 1).
	\end{aligned}
	\tag{DNWR-IE}
\end{equation}

A corresponding time discrete error estimate to the continuous one is also derived in \cite{timeDiscArticle} using the exponentially weighted discrete time Fourier transform. Denoting the Fourier frequency by $\omega$ and the weight parameter by $r$, the time discrete error estimate in \cite{timeDiscArticle} is then given by \begin{align}\label{eq:timeDiscEst}
	\|{\bf h}^k\|_{l^2} \leq \min_{r \geq 0} e^{rT} \max_{\omega \in [-\pi,\pi]} \left|\hat{\rho}^{IE}(r \Delta t + i \omega)\right|^{k} \norm{\bf{h}^0}_{l^2},
\end{align} where \begin{align}\label{eq:timeDiscRho}
\hat{\rho}^{IE}(r \Delta t+i\omega) = \hat{\rho}\left(\frac{1-e^{-r \Delta t - i \omega}}{\Delta t} \right).
\end{align}

To obtain the {\it fully discrete iteration}, we discretize the time discrete iteration in space using a standard centered finite difference scheme on two different grids with $n_1+1$ and $n_2+1$, equidistant grid points, respectively. Let $i_1$ and $i_2$ denote the spatial indices in the two domains, and let the grid sizes be $\Delta x_1 = a/n_1$ and $\Delta x_2 = b/n_2$, respectively. Furthermore, denote the fully discrete solution in domain $m=1,2$ at index $i_m$ and time step $t^n$ by $u^{n}_{m, i_m}$ and let ${\bf u}_{m,i_m} = (u^1_{m,i_m}, ..., u_{m,i_m}^{N})^T$.

To discretize the Neumann condition, a central difference scheme with one ghost point $u^{n,k}_{m,n_m+1}$ is used. The value in the ghost point is given by
\begin{equation}
	u^{n+1,k}_{m,n_m+1} = \frac{\alpha_m \Delta x_m^2}{\lambda_m \Delta t} \left(u^{n+1,k}_{m,n_m}-u^{n,k}_{m,n_m} \right) + 2u_{m,n_m}^{n+1,k} - u_{m,n_m-1}^{n+1,k}.
\end{equation} 
The fully discrete DNWR then reads: For $k=1,.....$,$n = 0,...,N_1$, $i_1 = 0, ...., n_1$ and $i_2  =  0, ...., n_2$, let
\begin{equation}\label{eq:DNWRFullyDisc}
	\begin{aligned}	
		\alpha_1 \frac{u^{n+1,k+1}_{1,i_1}-u^{n, k+1}_{1,i_1}}{\Delta t} -\frac{\lambda_1}{\Delta x_1^2} 	(u_{1,i_1+1}^{n+1, k+1}- 2u_{1,i_1}^{n+1, k+1} + u_{1,i_1-1}^{n+1, k+1}) &= f_{i_1}^{n+1}, \\ 
		u_{1,i_1}^{0,k+1} = g_{1,i_1},  \quad u^{n, k+1}_{1,0} &= u_{1,0}^{n+1}, \\
		u^{n+1,k+1}_{1,n_1} &= h^{n+1,k}, \\
		\alpha_2 \frac{u^{n+1,k+1}_{2,i_2}-u^{n,k+1}_{2,i_2}}{\Delta t} -\frac{\lambda_2}{\Delta x_2^2} (u_{2,i_2+1}^{n+1,k+1} - 	2u_{2,i_2}^{n+1,k+1} + u_{2,i_2-1}^{n+1,k+1}) &= f_{i_2}^{n+1},\\ 
		u_{2,i,0}^k = g_{2,i_2}, \quad 	u^{n+1,k+1}_{2,0} &= u^{n+1}_{2,0}, \\
		-\lambda_2 \frac{u^{n+1,k+1}_{2,n_2+1} - u^{n+1,k+1}_{2,n_2-1}}{2 \Delta x_2}+ \lambda_1 \frac{u^{n+1,k+1}_{1,n_1+1} - 	u^{n+1,k+1}_{1,n_1-1,j}}{2 \Delta x_1} &= 0, \\
		h^{n+1, k+1}(t) := (1 - \theta) h^{n+1,k} + \theta &u_{2,n_2}^{n+1, k+1}.
	\end{aligned}
	\tag{DNWR-FD}
\end{equation}  

%% file: fullyDiscreteEstimates.tex
\section{Error estimates for fully discrete DNWR}

To analyze the convergence behaviour of the fully discrete iteration \eqref{eq:DNWRFullyDisc}, we focus on the associated error equations, which are obtained by subtracting the solution from the iterates. Note that this is equivalent to setting $f_{i,j}$ and $g_{i,j}$ to zero. 

To derive error estimates, we use the exponentially weighted Fourier transform technique, which was proposed in \cite{Kwok2025} and first used in a time discrete setting in \cite{timeDiscArticle}. Exactly as in \cite{timeDiscArticle}, we define the exponentially weighted time discrete Fourier transform as:
\begin{definition}\label{def:DiscFourier}
	Let $r>0$ and let $\omega \in [-\pi,\pi]$ denote the Fourier frequency, then the exponentially weighted discrete time Fourier transform of the sequence $\{\phi^n\}_{n\in\mathbb{Z}}$ is defined as
	\begin{equation}\label{eq:transform}
		\widehat{\phi}(r \Delta t + i \omega) := \sum_{n\in\mathbb{Z}} \phi^n e^{-n (r \Delta t+i\omega)}, \quad \omega \in [-\pi, \pi],
	\end{equation}
	whenever the sum converges.
\end{definition}
The exponentially weighted time discrete Fourier transform of a sequence $\{\phi^n\}_{n\in\mathbb{Z}}$ is the time discrete Fourier transform, applied to the sequence $\{e^{-r n \Delta t} \phi^n\}_{n\in\mathbb{Z}}$, with the weight parameter $r>0$. We use the hat notation to denote such a transform, and write the argument as $r \Delta t + i \omega$ to emphasize the dependence on both the weight parameter $r$ and the frequency $\omega$. From Definition~\ref{def:DiscFourier}, we can verify that the exponentially weighted time discrete Fourier inherits properties such as linearity, time shifting, differentiation and Parseval's theorem. 
We refer to the monograph~\cite[Chapter 2.9]{Oppenheim1999} for a detailed introduction. 

To show that the transform exists for a class of functions which is relevant to our study, we use \cite[Lemma 3]{timeDiscArticle}, restated as Lemma \ref{lem:prepDisc2norm}.
\begin{lemma}\label{lem:prepDisc2norm}
	Let $r > 0$ and $\boldsymbol{\phi} = \{\phi^n\}_{n\in\mathbb{Z}}$ be a sequence such that $ \sum_{n\in\mathbb{N}} (\phi^n)^2 < \infty$, and $\phi^n=0$ for $n<0$. Then the exponentially weighted discrete time Fourier exists, and we have
	\[\frac{1}{2 \pi} \int_{-\pi}^{\pi} |\widehat{\phi}(r\Delta t+i\omega)|^2 d\omega \leq \norm{\boldsymbol{\phi}}_{l^2}^2.\]
\end{lemma}

To apply the transform \eqref{eq:transform} to ${\bf u}_{1,i_1}^{k},{\bf u}_{2,i_2}^{k}, {\bf h}^{k}$ in time, we need to extend them to $n \in (-\infty,\infty)$. With slight abuse of notation, we do this by setting \begin{align*}
	h^{n,k} = 0 \text{ for } n < 0 \text{ and } n>N \\
	u_{1,i}^{n,0} = 0 \text{ for } n < 0 \text{ and } n>N \\
	u_{1,i}^{n,0} = 0 \text{ for } n < 0 \text{ and } n>N.
\end{align*} The extension for $k>0$ is then defined through the iteration. 

Taking the transform \eqref{eq:transform} of the fully discrete iteration \eqref{eq:DNWRFullyDisc} yields 
\begin{equation}\label{eq:FourierFullyDisc}
	\begin{aligned}
		\frac{\alpha_1 \left(1 - e^{-r \Delta t - i \omega}\right) }{\Delta t} \hat{u}_{1,i_1}^{k+1} -\frac{\lambda_1}{\Delta x_1^2} (\hat{u}_{1,i_1+1}^{k+1}- 2 \hat{u}_{1,i_1}^{k+1} + \hat{u}_{1,i_1-1}^{k+1}) &= 0, \\ 
		\hat{u}^{k+1}_{1,0} &= 0,\\
		\hat{u}^{k+1}_{1,n_1} &= \hat{h}^{k}, \\
		\frac{\alpha_2 \left(1 - e^{-r \Delta t - i \omega}\right) }{\Delta t} \hat{u}_{2,i_2}^{k+1} -\frac{\lambda_2}{\Delta x_2^2} (\hat{u}_{2,i_2+1}^{k+1}- 2\hat{u}_{2,i_2}^{k+1} + \hat{u}_{2,i_2-1}^{k+1}) &= 0,\\ 
		\bar{u}^{k+1}_{2,0} &= 0,\\
		-\lambda_2 \frac{\hat{u}^{k+1}_{2,n_2+1} - \hat{u}^{k+1}_{2,n_2-1}}{2 \Delta x_2} + \lambda_1 \frac{\hat{u}^{k+1}_{1,n_1+1} - \hat{u}^{k+1}_{1,n_1-1}}{2 \Delta x_1} &=0,\\
		\hat{h}^{k} = (1-\theta) \hat{h}^{k-1} -  \theta &\hat{u}^k_{2,n_2}.
	\end{aligned}
\end{equation} 

Similar to \cite{clement2022} and \cite{wu2014}, we solve the difference equations exactly. The characteristic equation for both subproblems is given by \begin{align*}
	\frac{\alpha_m \left(1 - e^{-r \Delta t - i \omega}\right) }{\Delta t} \hat{u}_{m,i_m}^k(r \Delta t + i \omega) \\ -\frac{\lambda_m}{\Delta x_m^2} (\hat{u}_{m,i_m+1}^k(r \Delta t + i \omega)- 2\hat{u}_{m,i_m}^k(r \Delta t + i \omega) + \hat{u}_{m,i_m-1}^k(r \Delta t + i \omega)) = 0
\end{align*}  for $m= 1,2$, with solution 
\begin{align*}
	\hat{u}_{m,i_m}(r \Delta t + i \omega) = A{r_{m,+}(r \Delta t + i \omega)}^{i_m} + B{r_{m,-}(r \Delta t + i \omega)}^{i_m},
\end{align*}
where $A,B$ are constants and the function $r_{m, \pm}$ is given by
\begin{equation}\label{eq:rm}
	\begin{aligned}
		r_{m,\pm}(r \Delta t + i \omega) =& 1 + \frac{\alpha_m \left(1 - e^{-r \Delta t - i \omega} \right)}{2 \lambda_m} \frac{\Delta x_m^2}{\Delta t} 
		\\ &\pm \sqrt{\frac{ \alpha_m \left(1 - e^{-r \Delta t -i \omega} \right) }{ \lambda_m}\frac{\Delta x_m^2}{\Delta t}} \sqrt{1 + \frac{\alpha_m \left(1 - e^{-r \Delta t - i \omega} \right)}{4 \lambda_m}\frac{\Delta x_m^2}{\Delta t}}.
	\end{aligned}
\end{equation}

Imposing the boundary conditions yields the solution for the first subproblem as 
\begin{align*}
	\hat{u}^k_{1,i_1} = \frac{(r_{1,+})^{i_1}-(r_{1,-})^{i_1}}{(r_{1,+})^{n_1}-(r_{1,-})^{n_1}} \hat{h}^{k-1},
\end{align*} 
and for the second subproblem as
\begin{multline*}
	\hat{u}^k_{2,i_2} = \frac{- \lambda_1 \Delta x_2}{\lambda_2 \Delta x_1} \frac{(r_{2,+})^{i_2}-(r_{2,-})^{i_2}}{(r_{2,+})^{n_2+1}-
		(r_{2,-})^{n_2+1}-(r_{2,+})^{n_2-1}+(r_{2,-})^{n_2-1}}\\
	\frac{(r_{1,+})^{n_1+1}-(r_{1,-})^{n_1+1}-(r_{1,+})^{n_1-1}+(r_{1,-})^{n_1-1}}{(r_{1,+})^{n_1}-(r_{1,-})^{n_1}} \hat{h}^{k-1}.
\end{multline*}
Using the last equation in \eqref{eq:FourierFullyDisc}, we obtain by induction that
\begin{align*}
	\hat{h}^{k+1} = \left( (1 -\theta)- \theta \frac{\lambda_1 \Delta x_2}{\lambda_2 \Delta x_1} \frac{f_2}{f_1} \right)^{k} \hat{h}^0,
\end{align*} 
where the functions $f_m$, $m=1,2$, depend on $r\Delta t+i\omega$ and are defined as 
\begin{equation}\label{eq:fm}
	\begin{aligned}
		f_m = \frac{ (r_{m,+})^{n_m}-(r_{m,-})^{n_m}}{(r_{m,+})^{n_m+1}-(r_{m,-})^{n_m+1}-(r_{m,+})^{n_m-1}+(r_{m,-})^{n_m-1}}.
	\end{aligned}
\end{equation} 
The convergence factor $\hat{\rho}_{FD}^{IE}$ is then given by 
\begin{equation}\label{eq:fullyDiscConvfac}
	\hat{\rho}_{FD}^{IE}\left(r \Delta t + i \omega \right) = (1 -\theta)- \theta \frac{\lambda_1 \Delta x_2}{\lambda_2 \Delta x_1} \frac{f_2\left(r \Delta t + i \omega \right)}{f_1\left(r \Delta t + i \omega \right)}.
\end{equation} 
The error estimate for finite $T$ is then given by Theorem \ref{thm:superLinearEstFD}.
\begin{theorem}\label{thm:superLinearEstFD}
	The error of the fully discrete iteration \eqref{eq:DNWRFullyDisc} on the time interval $[0, T]$, satisfies \begin{align}\label{eq:fullyDiscEst}
		\|\boldsymbol{h}^k\|_{l^2} \leq \inf_{r \in \mathbb{R}^+} e^{rT}\max_{\omega \in [-\pi,\pi]} |\hat{\rho}^{IE}_{FD}(r \Delta t + i \omega)^k| \norm{\boldsymbol{h}^0}_{l^2}.
	\end{align}  
\end{theorem}
\begin{proof}
	As $r>0$, we have
	\begin{align*}
		\|\boldsymbol{h}^k\|_{l^2}^2 \leq
		e^{2 rT} \sum_{n=0}^{N} e^{-2 r \Delta t n} \left(h^{n, k}\right)^2 \leq e^{2 rT} \sum_{n=-\infty}^{\infty} \left( e^{- r \Delta t n}h^{n, k}\right)^2,
	\end{align*} where in the last inequality we use the extension of the sequence $\boldsymbol{h}^k$. By Lemma \ref{lem:prepDisc2norm} the exponentially weighted discrete time Fourier transform of $\boldsymbol{h}^k$ exists. Parseval's theorem then yields
	\begin{align*}
		e^{2rT} \sum_{n=-\infty}^{\infty} \left( e^{- r \Delta t n}h^{n, k}\right)^2
		&\leq e^{2rT} \frac{1}{2 \pi} \int_{-\pi}^{\pi} \left| \hat{h}^{k}(r \Delta t + i \omega) \right|^2  \\
		&\leq \frac{e^{2rT}}{2 \pi} \max_{\omega \in [-\pi,\pi]} \left|\hat{\rho}^{IE}_{FD}(r \Delta t + i\omega)\right|^{2k}\int_{-\pi}^{\pi} \left| \hat{h}^{0}(r \Delta t + i \omega) \right|^2  \\
		&\leq e^{2rT} \max_{\omega \in [-\pi,\pi]} \left|\hat{\rho}^{IE}_{FD}(r \Delta t + i\omega)\right|^{2k} \norm{\bf{h}^0}_{l^2}^2,
	\end{align*}
	where we use Lemma \ref{lem:prepDisc2norm} in the last inequality.  Lastly, taking the square root on both sides and the infimum over all $r$ yields \eqref{eq:fullyDiscConvfac}.
\end{proof}

Note that we obtain linear error estimates, corresponding to those in \cite{cllebl:22}, by setting $r=0$, resulting in: 
\begin{corollary}\label{cor:independentTEst}
	The error of the fully discrete iteration \eqref{eq:DNWRFullyDisc}, satisfies  \begin{align}\label{eq:independentTEst}
		\|\boldsymbol{h}^k\|_{l^2} \leq \max_{\omega \in [-\pi,\pi]} |\hat{\rho}^{IE}_{FD}(i \omega)|^k \norm{\boldsymbol{h}^0}_{l^2}.
	\end{align}  
\end{corollary}

\begin{remark}\label{rem:Tintuition}
	To give some intuition for the effect of $T$ on the estimate \eqref{eq:fullyDiscEst}, note that if $T$ is large, then $exp(rT)$ dominates in the fully discrete estimate \eqref{eq:fullyDiscEst}, and the optimal $r^*$ is close to zero, yielding linear convergence. However, if $T$ is small, then $exp(rT) \approx 1$, and the optimal $r^*$ is close to the $r$ minimizing \begin{equation}\label{def:supWfunc}
		F^{IE}_{FD}(r) = \max_{\omega \in [-\pi,\pi]}\left| \hat{\rho}^{IE}_{FD}\left(\frac{1-e^{-r \Delta t - i \omega}}{\Delta t}\right)\right|.
	\end{equation}
\end{remark}

%% file: analysisOfFullyDiscEst.tex
\section{Analysis of the fully discrete convergence estimate}\label{sec:3}

We now investigate the behaviour of the fully discrete estimate and discuss how it is qualitatively affected by the discretization, material parameters, domain sizes and time window size $T$. We begin by examining how the discretization, material parameters, and domain sizes influence $\hat{\rho}_{FD}^{IE}$ in \eqref{eq:fullyDiscConvfac}. 

Note that $\hat{\rho}_{FD}^{IE}$  depends linearly on the ratios $\lambda_1/\lambda_2$, $\Delta x_1/\Delta x_2$ and $\theta$. However, $\hat{\rho}_{FD}^{IE}$ also depends on the functions $f_1$ and $ f_2$, which in turn depend nonlinearly on the number of gridpoints $n_1$ and $n_2$, as well as on
 \begin{equation}\label{def:sm}
 s_m = \frac{\alpha_m \left(1 - e^{-r \Delta t - i \omega} \right)}{2 \lambda_m} \frac{\Delta x_m^2}{\Delta t}, \quad \text{ for } m=1,2.
 \end{equation} 
 By approximating $\hat{\rho}_{FD}^{IE}$ in the regimes $|s_m| \ll 1$ and $|s_m| \gg 1$, we obtain Theorem \ref{thm:assymptoticsRhoFD}.

\begin{theorem}\label{thm:assymptoticsRhoFD}
	If $|s_1|\gg 1$ and $|s_2| \gg 1 $, then \[\hat{\rho}_{IE}^{FD}(r \Delta t + i \omega) \approx 1 - \theta - \theta \frac{\alpha_1 \Delta x_1}{\alpha_2 \Delta x_2}.\]
	If $|s_1|\ll 1$ and $|s_2| \ll 1 $, then \begin{align*}
		\hat{\rho}_{FD}^{IE}(r \Delta t + i \omega) \approx \hat{\rho}^{IE}(r\Delta t+i\omega) = \\
		1 - \theta - \theta\sqrt{\frac{\alpha_1 \lambda_1}{\alpha_2 \lambda_2}} \coth\left(\sqrt{\frac{\alpha_1 (1-e^{- r \Delta t-i\omega})}{\lambda_1 \Delta t}}a\right)\tanh\left(\sqrt{\frac{\alpha_2 (1-e^{- r \Delta t-i\omega })}{\lambda_2 \Delta t}}b\right),
	\end{align*}
	where $\hat{\rho}^{IE}$ is given by \eqref{eq:timeDiscRho}.
\end{theorem}
\begin{proof}
	If $s_m \gg 1$ for $m=1,2$ then $r_{m,\pm}$ given in \eqref{eq:rm} can be approximated by \begin{align*}
		r_{m, +} &= 1 + s_m/2 + \sqrt{s_m} \sqrt{1 + s_m/4}  \\
		&\approx 1 + s_m/2 + \sqrt{s_m} \sqrt{s_m/4 } 
		\approx s_m, 
	\end{align*}
	and
	\begin{align*}
		r_{m, -} &= 1 + s_m/2 - \sqrt{s_m} \sqrt{1 + s_m/4}  \\ &\approx 1 + s_m/2 - \sqrt{s_m} \sqrt{s_m/4 } 
		\approx 1.
	\end{align*}
	Inserting the approximation into $f_m$, defined in \eqref{eq:fm}, yields 
	\[f_m = \frac{(r_{m,+})^{n_m}-(r_{m,-})^{n_m}}{(r_{m,+})^{n_m+1}-(r_{m,-})^{n_m+1}-(r_{m,+})^{n_m-1}+(r_{m,-})^{n_m-1}} \approx \frac{1}{s_m}, \]
	resulting in \[
	\hat{\rho}_{FD}^{IE} =  (1 -\theta)- \theta \frac{\lambda_1 \Delta x_2}{\lambda_2 \Delta x_1} \frac{f_2}{f_1} \approx 1 - \theta - \theta \frac{\alpha_1 \Delta x_1}{\alpha_2 \Delta x_2}.
	\] 
	
	For the second case where $s_1 \ll 1$ and $s_2 \ll 1$, $r_{m,\pm}$ can instead be approximated by \begin{align*}
		r_{m, \pm} = 1 + s_m/2 \pm \sqrt{s_m} \sqrt{1 + s_m/4} &\approx 1 + s_m/2 \pm \sqrt{s_m} \\ 
		&\approx 1 \pm \sqrt{s_m}.
	\end{align*}
	Next, we rewrite $f_m$, yielding 
	\begin{equation*}
		f_m = \frac{e^{n_m \ln(r_{m,+})}-e^{n_m  \ln(r_{m,-})} }{\left(e^{(n_m+1) \ln(r_{m,+})} - e^{(n_m+1)\ln(r_{m,-})}\right) + (e^{(n_m-1)\ln(r_{m,+})} - e^{(n_m-1)\ln(r_{m,-})})}.
	\end{equation*} 
	We insert the approximation of $r_{m,\pm}$ into $\ln(r_{m,\pm})$ and do a first order Taylor expansion in $\sqrt{s_m}$ yielding 
	\begin{equation*}
		\ln(r_{m,\pm}) \approx -\pm \sqrt{s_m}.
	\end{equation*} 
	Inserting this into $f_m$ and simplifying yields 
	\begin{align*}
		f_m \approx \frac{(e^{n_m\sqrt{s_m}} - e^{-n_m\sqrt{s_m}})}{(e^{n_m\sqrt{s_m}} + e^{-n_m\sqrt{s_m} }) (e^{\sqrt{s_m}} - e^{\sqrt{s_m}})}
		= \frac{\tanh\left(n_m \sqrt{s_m}\right)}{ e^{\sqrt{s_m}} - e^{\sqrt{s_m}}}.
	\end{align*} 
	 Approximating the exponential terms by $e^{\pm\sqrt{s_m}} \approx 1\pm\sqrt{s_m}$, yields
	 	\begin{align*}
	 	f_m &\approx \frac{\tanh\left(n_m \sqrt{s_m}\right)}{1 + \sqrt{s_m} - 1 + \sqrt{s_m}} = \frac{\tanh\left(n_m \sqrt{s_m}\right)}{2 \sqrt{s_m}}.
	 \end{align*} 
	Inserting the approximation of $f_m$ into $\hat{\rho}_{FD}^{IE}$ and simplifying yields the result.
\end{proof}

\subsection{Discussion of Theorem \ref{thm:assymptoticsRhoFD}}

Theorem \ref{thm:assymptoticsRhoFD} provides insight into the dependence of the estimate \eqref{eq:fullyDiscEst} on the material parameters, domain lengths and discretizations. First, note that $\left|s_m\right|$ is bounded by 
\[\left| s_m \right| = \left| \frac{\alpha_m \left(1 - e^{-r \Delta t - i \omega}\right) \Delta x_m^2}{2\lambda_m \Delta t}\right| \leq \frac{\alpha_m \Delta x_m^2}{\lambda_m \Delta t}=:\frac{1}{c_m}.\]
The number on the right hand side is the inverse of the CFL number chosen in the respective solver, which we now denote by $c_m$. We thus distinguish the two cases of the CFL numbers $c_m$ being small or large in Remarks \ref{rem:small} and \ref{rem:large}, respectively. 

\begin{remark}\label{rem:small}
	For $c_m\gg 1$, $m=1,2$, the spatial discretization does not affect the convergence behaviour: If $c_m\gg 1$  then $\left| s_m\right| \ll 1$, thus Theorem \ref{thm:assymptoticsRhoFD} yields that $\hat{\rho}_{FD}^{IE}$ is close to $\hat{\rho}^{IE}$, and hence that the fully discrete estimate~\eqref{eq:fullyDiscEst} is close to the time discrete estimate~\eqref{eq:timeDiscEst}.
\end{remark}
\begin{remark}\label{rem:large}
	If $c_m \ll 1$, then we expect the fully discrete estimates to converge linearly with convergence factor 
	\begin{align}\label{eq:approx}
		\begin{aligned}
			\inf_{r \in \mathbb{R}^+} e^{rT} \max_{\omega \in [-\pi,\pi]} | \hat{\rho}^{IE}_{FD}(r \Delta t + i \omega)^k| &\approx \inf_{r \in \mathbb{R}^+} e^{rT} \max_{\omega \in [-\pi,\pi]}  \left|1 - \theta - \theta \frac{\alpha_1 \Delta x_1}{\alpha_2 \Delta x_2}\right|^k\\
			&= \left|1 - \theta - \theta \frac{\alpha_1 \Delta x_1}{\alpha_2 \Delta x_2}\right|^k,
		\end{aligned}
	\end{align}
	where we used Theorem \ref{thm:assymptoticsRhoFD} to obtain the approximation. This implies that space discretization plays an important role. 
\end{remark}

Note that $c_m$ being small or large does not necessarily imply that the discretizations are fine or coarse. First, $\alpha_m/\lambda_m$ can be large or small, and second, one can have small CFL numbers for fine discretizations.  

Theorem \ref{thm:assymptoticsRhoFD} also explains how $\hat{\rho}_{FD}^{IE}$ relates to the fully discrete one time step estimates from \cite{monbir:18}. In particular, if $c_m$, $m=1,2$, is small, then for large frequencies $\omega$, $\hat{\rho}_{IE}^{FD} \approx 1 - \theta - \theta \frac{\alpha_1 \Delta x_1}{\alpha_2 \Delta x_2}$. This expression is consistent with the limit $\Delta t \to 0$ of the convergence rate for one implicit Euler time step with finite element discretization in \cite{monbir:18}. Furthermore, if $c_m$ is large, then $\hat{\rho}_{FD}^{IE} \approx \hat{\rho}^{IE}$. If in addition, $\Delta t$ is large, then \begin{align*}
	|\hat{\rho}^{IE}(r\Delta t +i\omega)| &= 1-\theta -\theta \sqrt{\frac{\alpha_1\lambda_1}{\alpha_2\lambda_2}}\frac{\tanh\left(\sqrt{\frac{\alpha_2 \frac{1-e^{-r \Delta t - i \omega}}{\Delta t}}{\lambda_2}}b\right)}{\tanh\left(\sqrt{\frac{\alpha_1 \frac{1-e^{-r \Delta t - i \omega}}{\Delta t}}{\lambda_1}}a\right)} \\ &\approx |1 - \theta - \theta \frac{\lambda_1 b}{\lambda_2 a}|,
\end{align*} since $\left(1-e^{-r \Delta t - i \omega}\right)/\Delta t$ is close to $0$ when $\Delta t$ is large independently of $r$ and $\omega$. This result is also consistent with the limit $\Delta x \to 0$ of the convergence rate for one implicit Euler time step with finite element discretizations in \cite{monbir:18}.

\subsection{Case studies}

Next, we investigate the accuracy of Theorems \ref{thm:superLinearEstFD} and \ref{thm:assymptoticsRhoFD} using a simple test case where all parameters are set to one except for $\alpha_1 = 4$ and $\theta = 1/3$. Note that the continuous estimate predicts superlinear convergence for $\theta=1/3$. 

We now plot the dependence of $|\hat{\rho}_{FD}^{IE}|$ on $\omega$ for $r=0$ in Figure \ref{fg:simpleRhoW}. 
We set $\Delta x_1 =\Delta x_2 = 0.1,0.01$ in combination with $\Delta t = 0.1,0.0001$, resulting in $c_1$ varying between $1/200$ and $5\cdot 10^4$, and  $c_2$ varying between $1/800$ and $1.25 \cdot 10^{4}$. As a comparison we also plot $ |\hat{\rho}^{IE}|$. 

\begin{figure}[t]
	\centering
	\mbox{
		\includegraphics[width=0.5\textwidth]{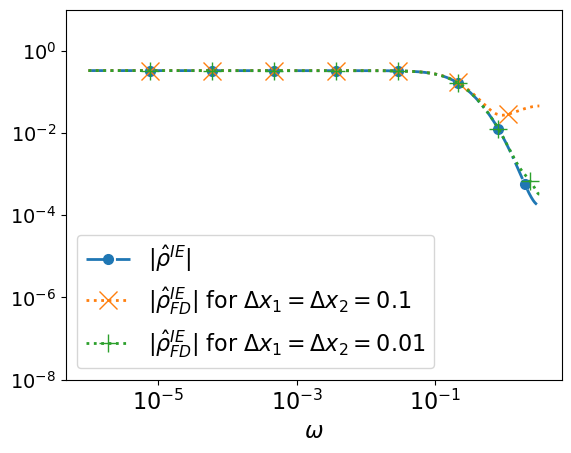}
		\includegraphics[width=0.5\textwidth]{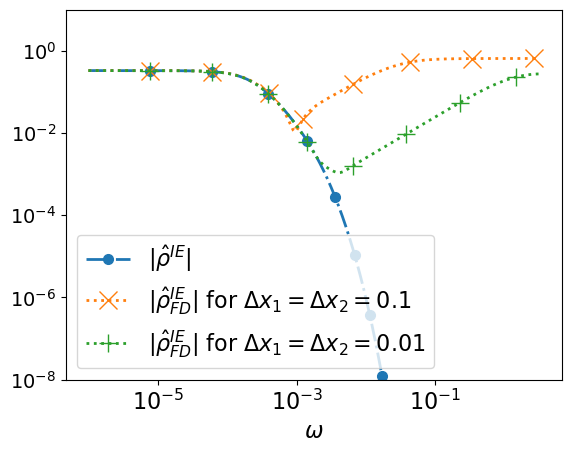}
	}
	\caption{The fully-discrete and time-discrete convergence factors dependence on $\omega$ for $r=0$ and $\theta = 1/3$ with all material parameters equal to one except for $\alpha_1 = 4$. From left to right with $ \Delta t = 0.1, 0.001$.}
	\label{fg:simpleRhoW}
\end{figure}

The left plot of Figure \ref{fg:simpleRhoW} shows the dependence of $|\hat{\rho}_{FD}^{IE}|$ on $\omega$, with $c_1 = 2.5 \cdot 10^4$ and $1 \cdot 10^{5}$ for $\Delta x = 0.01$, and $c_1 = 2.5$ and $c_2 = 10$ for $\Delta x = 0.1$. For the large CFL numbers, $|\hat{\rho}_{FD}^{IE}|$ and $|\hat{\rho}^{IE}|$ are almost identical, consistent with theorem~\ref{thm:assymptoticsRhoFD}. Decreasing the CFL number s.t. the theorem no longer applies leads to $|\hat{\rho}_{FD}^{IE}|$ deviating from $|\hat{\rho}^{IE}|$. From \eqref{def:sm}, we can infer that the values of $s_m$ will decrease first for large frequencies, which we see in the plot. 


In the right plot of Figure \ref{fg:simpleRhoW}, $c_m$ varies from $1/2$ to $1/800$ for $m=1,2$. For $\Delta x_m=0.1$, we are in the regime of small CFL numbers. Based on Theorem \ref{thm:assymptoticsRhoFD}, we therefore expect  $|\hat{\rho}_{FD}^{IE}|$ to be close to $|\hat{\rho}^{IE}|$ for small $\omega$ and to approach $|1-\theta-\theta \alpha_1 \Delta x_1/(\alpha_2 \Delta x_2)| = 2/3$ as $\omega$ becomes large, which is consistent the plot. Increasing the CFL numbers by decreasing $\Delta x_m$ to 0.01 means that theorem \ref{thm:assymptoticsRhoFD} no longer applies, and we observe that, again starting with high frequencies, the curve moves towards the semidiscrete convergence factor. 

\begin{figure}[t]
	\centering
	\mbox{
		\includegraphics[width=0.5\textwidth]{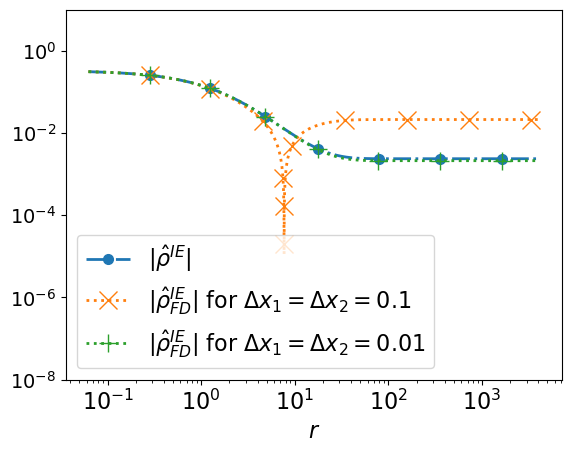}
		\includegraphics[width=0.5\textwidth]{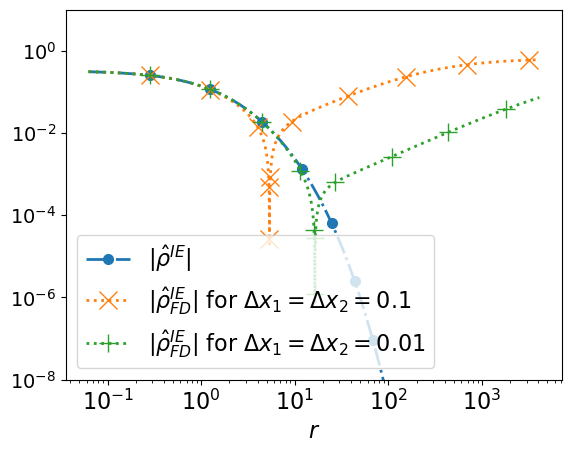}
	}
	\caption{The fully-discrete and time-discrete convergence factors dependence on $r$ for $\omega =0$ and $\theta = 1/3$ with all material parameters equal to one except for $\alpha_1 = 4$. From left to right with $ \Delta t = 0.1, 0.001$.}
	\label{fg:simpleRhoR}
\end{figure}

Figure \ref{fg:simpleRhoR}, shows the dependence of $|\hat{\rho}_{FD}^{IE}|$ on $r$. In the left plot we have again $c_1 = 2.5\cdot 10^4, c_2 = 1\cdot 10^5$ ($\Delta x=0.01$) or $2.5$ and $10$ ($\Delta x=0.1$) for $m=1,2$. Again, for the small CFL number, $|\hat{\rho}_{FD}^{IE}|$ does not agree with $|\hat{\rho}^{IE}|$, in particular for $r$ large. This is expected from theorem \ref{thm:assymptoticsRhoFD}, since for $\omega=0$, $s_m$, grows with $r$. We observe that the function has a zero in $r$, which is due to the convergence factor changing sign. 

In the right plot of Figure \ref{fg:simpleRhoW}, $\Delta t$ is again decreased to 0.001, leading to smaller CFL numbers. Based on Theorem \ref{thm:assymptoticsRhoFD}, we therefore expect  $|\hat{\rho}_{FD}^{IE}|$ to approach $|1-\theta-\theta \alpha_1 \Delta x_1/(\alpha_2 \Delta x_2)| = 2/3$ first for large $r$, which is consistent the plot.

%% file: experiments.tex
\section{Numerical experiments}\label{sec:numExp}

We now compare the continuous \eqref{eq:contEst}, time discrete \eqref{eq:timeDiscEst} and fully discrete \eqref{eq:fullyDiscEst} estimates with the observed convergence behavior of the fully discrete iteration \eqref{eq:DNWRFullyDisc}. To this end, we employ the error equations, where the boundary conditions and inital data are set to zero, except for the initial waveform error $\boldsymbol{h}^0$. We initialize $\boldsymbol{h}^0$ randomly, see ~\cite[Section 5.1]{gander} for the importance of this. To compute the error estimate, the min-max problem is solved numerically using the Nelder-Mead optimization routine from SciPy with an absolute tolerance of $10^{-6}$. 

For the cases where the iteration converges linearly, we estimate the convergence rate as the average of the ratio $\|{\bf h}^{k+1}\|_{l_2} / \|{\bf h}^{k}\|_{l_2}$ over the first five iterations.

We consider the following two test cases.
\begin{itemize}
	\item \textbf{Simple:} Material parameters $\alpha_1 = 4$ and $\alpha_2 = \lambda_1 = \lambda_2 = 1$, with unit domain lengths. We set $\theta = \theta^* = 1/3$, since the continuous estimate predicts superlinear convergence for this choice of relaxation parameter.
	
	\item \textbf{Air-steel:} A model of minimal relevance for thermal FSI, or more specifically gas quenching, where a hot piece of steel is cooled by an airstream. This test case uses realistic material parameters $\alpha = 1299$ and $\lambda = 0.0243$ for air, and $\alpha = 3.47 \times 10^6$ and $\lambda = 48.9$ for steel from \cite{monbir:18, meimonbir:23}. The domain sizes are chosen to reflect realistic scales: the steel domain has width $1.5 \cdot 10^{-3}$, while the air domain has width $0.1$. In addition, we set  $\theta = \theta^* = 0.9996$, since the continuous estimates predicts super linear convergence for this choice of relaxation parameter.
	
\end{itemize}

Note that the ratios $\alpha_1/\lambda_1$ and $\alpha_2/\lambda_2$ for case {\bf Air-Steel} are $5.3 \cdot 10^4$ and $7.0 \cdot 10^4$, respectively. 
Thus, $c_m \ll 1$ unless huge time steps are used. By Remark~\ref{rem:small}, we thus expect to observe mesh dependent linear convergence. 

In contrast, for case {\bf Simple} the ratios $\alpha_1/\lambda_1$ and $\alpha_2/\lambda_2$ are $1$ and $1/4$, respectively. Thus, $c_m$ depends strongly on the subsolvers meshes. 

We first consider test case \textbf{Simple} with the grid sizes $\Delta x_1 =\Delta x_2 = 0.01, 0.0025$ in combination with $\Delta t = 0.1, 10^{-6}$. In addition, we set $T=1$, since the continuous estimate predicts superlinear convergence for this $T$.
\begin{figure}[t]
	\centering
		\includegraphics[width=1.0\textwidth]{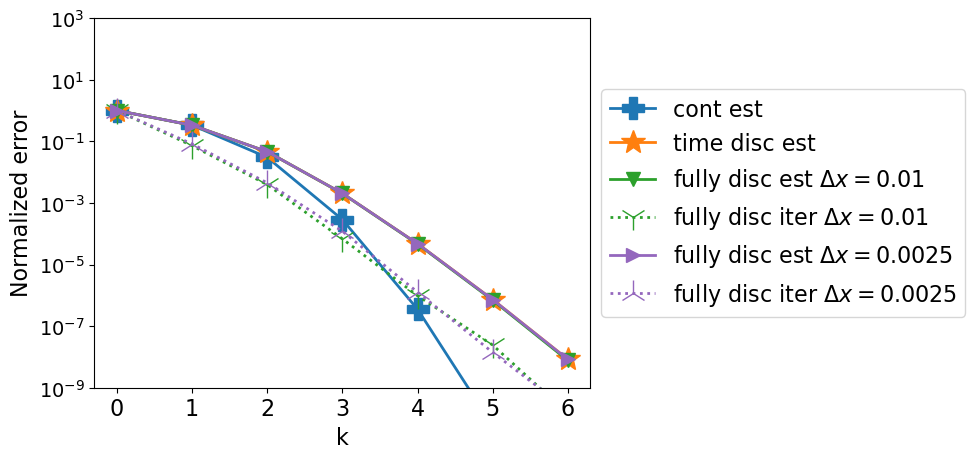}
	\caption{The normalized error $\|{\bf h}^k\|_2/\|{\bf h}^0\|$, and the continuous, time discrete and fully discrete error estimates for case \textbf{Simple}, with $T=1.0$, $\Delta t = 0.1$ and $\Delta x_1=\Delta x_2= \Delta x$.}
	\label{fg:simpleRes1}
\end{figure}

Figure \ref{fg:simpleRes1} shows the relative error $\|{\bf h}^k\|_2/\|{\bf h}^0\|$ together with the continuous, time discrete and fully discrete error estimates for the simple test case and $\Delta t = 0.1$. Comparing the relative error with the fully discrete estimate, we see that they are close. Furthermore, we see that both are almost unchanged when varying $\Delta x$. In addition, we see that the time discrete and the fully discrete estimates for $\Delta x=0.01,0.0025$ are almost identical. Both of these observations are explained by Theorem \ref{thm:assymptoticsRhoFD} and the fact that the CFL number in the subsolvers varies from $10^4$ to $6.4\cdot10^{4}$, which is large.

Additionally, we see that these two estimates have a superlinear behavior for the first few iterations, which then quickly becomes linear. This was proven for the timediscrete estimate in \cite{timeDiscArticle}. In contrast, the continuous estimate converges superlinearly and thus has to deviate from the other two after a few iterations, as also proven in  \cite[Theorem 2.5 and 2.6]{timeDiscArticle}. 

\begin{figure}[t]
	\centering
	\mbox{
		\includegraphics[width=1.0\textwidth]{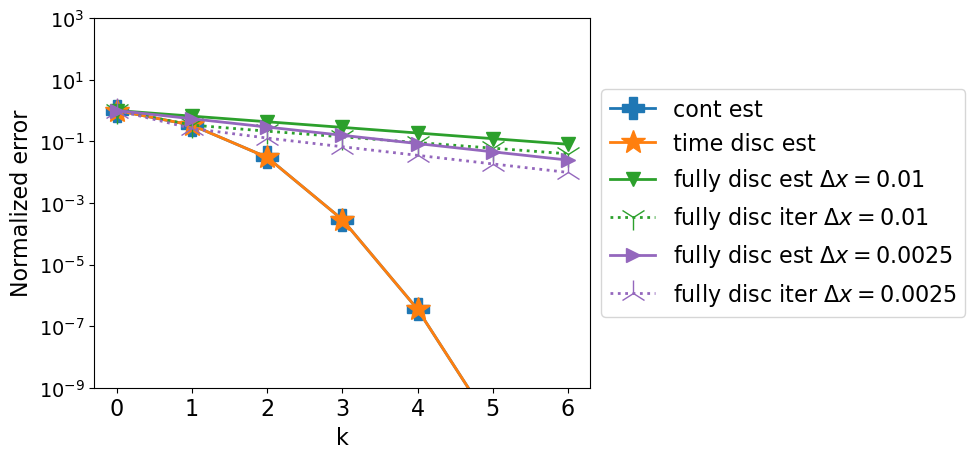}}
	\caption{The normalized error $\|{\bf h}^k\|_2/\|{\bf h}^0\|$, and the continuous, time discrete and fully discrete error estimates for case \textbf{Simple}, with $T=1.0$ and $\Delta t = 1 \cdot 10^{-6}$ and $\Delta x_1=\Delta x_2= \Delta x$.}
\label{fg:simpleRes2}
\end{figure}

In Figure \ref{fg:simpleRes2}, we decrease the time step to $\Delta t = 10^{-6}$. This causes the time discrete estimate to be almost identical to the continuous estimate, consistent with corollary 4.2 from \cite{timeDiscArticle}. We also see that the fully discrete estimate and the relative error converge linearly and that the fully discrete estimate still accurately predicts the error.

Note that decreasing the time step to $\Delta t = 10^{-6}$ reduces the subsolvers’ CFL number to $6.4$ and $0.1$. Thus, Theorem \ref{thm:assymptoticsRhoFD} no longer applies and instead we are in a transition region between the asymptotic behaviour for small and large CFL numbers. We can see that the smaller $\Delta x$ (larger CFL number) leads to the fully discrete curves getting closer to the time discrete one. Decreasing $\Delta x$ even further would lead to Theorem \ref{thm:assymptoticsRhoFD}  holding again, and the time discrete and fully discrete estimates to be close. 

Calculating the convergence rate of the fully discrete estimate and the normalized error yields $0.48$ and $0.54$ for $\Delta x = 0.0025$, respectively, and $0.60$ and $0.65$ for $\Delta x = 0.01$, respectively. Comparing with the approximate converge rate for small CFL numbers from Remark \ref{rem:large}, which for this case equals $2/3$, we see the convergence rates are close for $\Delta x = 0.01$, suggesting that we see the behaviour for small CFL numbers.

Next, we use case \textbf{Air-Steel}. We deploy a fine and a coarse grid with $\Delta x = 10^{-3}, 10^{-6}$. Both the air and the steel domain are discretized using the fine and coarse grid, yielding four different configurations. Additionally, we set $\Delta t = 10^{-3}$ and $T=1$. 

\begin{figure}[t]
	\centering
	\includegraphics[width=1.0\textwidth]{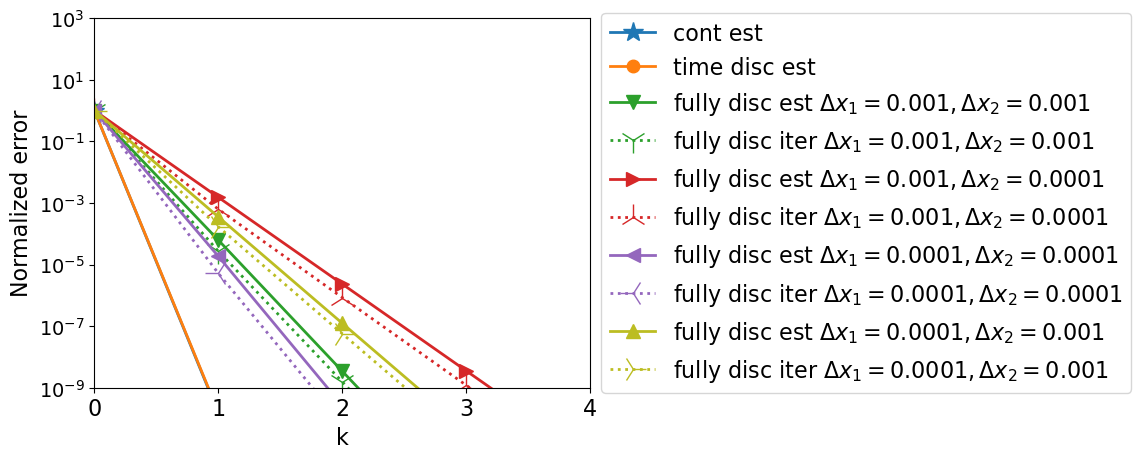}
	\caption{The fully discrete iteration and the continuous and discrete error estimate for test case {\bf Air-Steel} with $ \Delta t = 10^{-3}$ and $T=1$. }
	\label{fg:AirSteelSim}
\end{figure}

Figure \ref{fg:AirSteelSim} presents the normalized error along with the corresponding error estimates for different grid sizes $\Delta x$. We see that the continuous and time discrete estimates yield very fast convergence, reaching an error of less than $10^{-9}$ in one iteration. Comparing with the the normalized error, we see that it converges slower for all grid sizes $\Delta x$, which is accurately captured by the fully discrete estimate. 

This is explained by the large CFL numbers of both subsolvers. By Remark \ref{rem:large}, the fully discrete estimate can be approximated by \eqref{eq:approx}, and we therefore expect linear convergence with a rate depending on the grid size $\Delta x$, consistent with Figure \ref{fg:AirSteelSim}.

\begin{table}
	\begin{center}
		\begin{tabular}{ |c|c|c| } 
			\hline
			 & Conv. rate iter. &Est. conv. rate  \eqref{eq:approx}  \\ 
			 \hline
			$\Delta x_1 = 10^{-3}, \Delta x_2=10^{-3}$ & $4 \cdot 10^{-5}$ & $6 \cdot 10^{-5}$ \\ 
			\hline
			$\Delta x_1 = 10^{-3}, \Delta x_2=10^{-4}$ & $2 \cdot 10^{-3}$ & $3 \cdot 10^{-3}$ \\
			\hline
			$\Delta x_1 = 10^{-4}, \Delta x_2=10^{-3}$ & $3 \cdot 10^{-4}$ & $4 \cdot 10^{-4}$ \\ 
			\hline
			$\Delta x_1 = 10^{-4}, \Delta x_2=10^{-4}$ & $3 \cdot 10^{-5}$ & $6 \cdot 10^{-5}$ \\ 
			\hline
		\end{tabular}
	\end{center}
	\caption{ The convergence rate of the iteration and the estimated convergence rate \eqref{eq:approx} for test case {\bf Air-Steel} with $ \Delta t = 10^{-3}$ and $T=1$.  \label{tab:convRates}}
\end{table}

Comparing the convergence rates of the iteration with the estimated convergence rate \eqref{eq:approx} for the different meshes in Table \ref{tab:convRates}, we see that the estimated convergence rates are close, with the largest difference being a factor $2$ for $\Delta x_1 = \Delta x_2 = 10^{-4}$. This suggest that for realistic parameters and mesh sizes, where the CFL numbers $c_m$ are small, we observe linear mesh dependent convergence with the convergence rates being close to the estimates \eqref{eq:approx}.